\documentclass[a4paper,12pt]{article}
\usepackage{latexsym}
\usepackage{amsmath}
\usepackage{amssymb}
\usepackage{enumerate}

\usepackage{theorem}
\newtheorem{theorem}{Theorem}[section]
\newtheorem{proposition}[theorem]{Proposition}
\newtheorem{lemma}[theorem]{Lemma}
\newtheorem{corollary}[theorem]{Corollary}

\theorembodyfont{\rmfamily}
\newtheorem{proof}{\textmd{\textit{Proof.}}}

\newtheorem{remark}[theorem]{Remark}

\makeatletter

\@addtoreset{equation}{section}
\makeatother

\newcommand{\qedd}{\hfill \Box}
\newcommand{\ve}{\varepsilon}
\newcommand{\lra}{\longrightarrow}
\newcommand{\wt}{\widetilde}

\newcommand{\R}{\ensuremath{\mathbb{R}}}

\def\supp{\mathop{\mathrm{supp}}\nolimits}

\title{The topology of an open manifold with radial curvature bounded from below 
by a model surface with finite total curvature and examples of model surfaces\footnote{
Mathematics Subject Classification (2010)\,:\,53C21, 53C22.}
}
\author{Minoru TANAKA $\cdot$ Kei KONDO}
\date{}
\begin{document}


\maketitle

\begin{abstract}
We will construct peculiar surfaces of revolution with finite total curvature whose Gauss 
 curvatures are not  bounded. 
Such a surface of revolution is employed as a reference surface of comparison theorems in radial curvature geometry. 
Moreover, we will prove that a complete non-compact Riemannian manifold $M$ is homeomorphic to 
the interior of a compact manifold with boundary, if the manifold   $M$   is not less curved than a non-compact model surface $\wt{M}$ of revolution, and if  the total  curvature of the model surface 
$\wt{M}$  is finite and less than $2\pi$.\par 
By the first result mentioned above, the second result covers a much wider class of manifolds than that of complete non-compact Riemannian manifolds whose sectional curvatures 
are bounded from below by a constant.
\end{abstract}

\section{Introduction}

In a series of our articles (\cite{KT1}, \cite{KT2}, and \cite{KT3}),
by restricting the total curvature of a non-compact model surface of revolution,
we investigated some topological properties of a complete and non-compact  Riemannian manifold which is not less curved than the model surface. The precise definition
to be ``not less curved than a non-compact model surface of revolution" will be defined later.
Typical non-compact model surfaces are Euclidean plane
 $(\R^2,dt^2+t^2d\theta^2)$ and a hyperbolic plane $(\R^2,dt^2+\sinh^2td\theta^2).$
Here $(t,\theta)$ denotes polar coordinates around the origin of $\R^2.$
A non-compact model surface of revolution $(\wt{M}, \tilde{p})$ will be constructed as follows:
Let  a smooth function $f:(0,\infty)\lra (0,\infty) $  be  given. Then, 
$(\R^2,dt^2+f(t)^2d\theta^2)$ is a non-compact complete surface of revolution $\wt{M}$ 
with smooth Riemannian metric $dt^2+f(t)^2d\theta^2$ around the base point $\tilde{p} \in \wt{M}$, 
if $f$ is extensible to a smooth odd function around $0$ and satisfies $f'(0)=1$ 
(see \cite[Theorem 7.1.1]{SST}). 
It is well-known that the Gauss curvature $G$ of $\wt M$ is given by 
\[
G(q)=-\frac{f''}{f}(t(q)).
\]
The total curvature $c(\wt M)$ of  a non-compact model surface of revolution
$\wt M$ is defined by
\[
c(\wt M):= \int_{\wt M}G_+ d\wt M    + \int_{\wt M}G_- d\wt M,
\]
if $ \int_{\wt M}G_+ d\wt M <\infty $ or  $ \int_{\wt M}G_- d\wt M>-\infty.$
Here $G_+:=\max\{G,0\},$ $G_-:=\min\{G,0\}$ and $d\wt M$ denotes the area element of $\wt M.$ 
The total curvature of a complete 2-dimensional  Riemannian manifold is defined analogously.
This definition was introduced by Cohn-Vossen.\par
In 1935, Cohn-Vossen generalized the Gauss-Bonnet theorem for non-compact 
Riemannian manifolds:

\begin{theorem}{\rm (\cite{CV})}\label{thmCV}
If a connected, complete non-compact, finitely-connected Riemannian $2$-dimensional 
manifold $X$
admits a total curvature $c(X)$, 
then 
\[
c(X)\leq 2\pi\chi(X)
\]
holds. Here $\chi(X)$ denotes the Euler characteristic of $X$.
\end{theorem}

\bigskip

Now, we are in a position to give the precise definition to be ``not less curved 
than a non-compact model surface of revolution": 
Let $(M,p)$ denote a complete, connected and  non-compact $n$-dimensional 
Riemannian manifold with base point $p\in M$ and $(\wt M,\tilde p)$ a non-compact model 
surface of revolution defined above. Let us note that a unit speed geodesic 
$\wt\gamma :[0,\infty)\lra \wt M$  emanating from $\tilde p$, 
which is called a {\it meridian,} is a ray. 
From now on, we choose a meridian $\wt \gamma$ and fix it. 
We say that the manifold $(M,p)$ has radial curvature at the base point $p$ 
bounded from below by that of the model surface $(\wt M, \tilde p)$, 
if along every minimal geodesic $\gamma :[0,a)\lra M$ emanating from $p=\gamma(0)$, 
its sectional curvature $K_M$ satisfies 
\[
K_M(\sigma_t)\geq G(\wt\gamma(t))
\]
for all $t\in[0,a)$ and $2$-dimensional linear planes $\sigma_t$ containing $\gamma'(t)$. 
This is the precise definition that a complete non-compact Riemannian manifold is not less curved than a model surface.

\bigskip

By  Theorem \ref{thmCV}, the total curvature of a non-compact model surface  of revolution does not exceed $2\pi,$ if the total curvature exists. Hence it is natural  to assume that the total curvature of a non-compact model surface of revolution is finite. 
Under this assumption we have proved  the following theorem. 
 
 \begin{theorem}{\rm (\cite[Theorem 2.2]{KT2})}\label{thm4.3}
Let $(M,p)$ be a complete non-compact Riemannian manifold $M$ 
whose radial sectional curvature at the base point $p$ is bounded from below by
that of a non-compact model surface of revolution $(\wt{M}, \tilde{p})$ 
with its metric $dt^2 +  f(t)^2d \theta^2$. If 
\begin{enumerate}[{\rm ({A--}1)}]
\item
$\wt{M}$ admits a finite total curvature, and 
\item
$\wt{M}$ has no pair of cut points in a sector $\wt{V} (\delta_{0})$ for some $\delta_{0} \in (0, \pi]$, 
\end{enumerate}
then $M$ is homeomorphic to the interior of a compact manifold with boundary. 
Here $\wt{V}(\delta_{0}) := \{ \tilde{x} \in \wt{M} \, | \, 0 < \theta(\tilde{x}) < \delta_{0} \}$.
\end{theorem}

In this article, 
we will show that the assumption (A--2) of Theorem \ref{thm4.3} is unnecessary 
if the total curvature is less than $2\pi.$ That is, we will prove the following theorem:

\begin{theorem}\label{thm1.3}
A connected, complete, non-compact Riemannain manifold $(M, p)$ is homeomorphic to the interior of a compact manifold with boundary if the radial curvature at a point $p\in M$ is bounded from below 
by that of a non-compact model surface of revolution $(\wt{M}, \tilde{p})$ 
which admits a finite total curvature $c(\wt M)$ less than $2\pi$.
\end{theorem}

\noindent
Note that the finiteness of the total curvature does not impose strong restriction  on the curvature of the model surface. In fact, we will prove  the following theorem which tells us that the radial curvature of the model surface in Theorem \ref{thm1.3} is not always bounded from below.

\begin{theorem}\label{thm1.4}
Let $\wt{M}:=(\R^2,dt^2+f(t)^2d\theta^2)$ denote a non-compact model surface of revolution 
which admits a finite total curvature $c(\wt M)$ less than $2\pi.$ Then, for any $\ve>0,$ there exists a non-compact model surface of revolution $\wt{M}_{\ve}^{-}:=(\R^2,dt^2+m_\ve^-(t)d\theta^2)$ 
such that
\[
K\geq G_\ve^- \ {\it on} \  [0,\infty),
\]
\[
||G_\ve^{-}-K||_2<\ve,
\]
\[
\liminf_{t\to\infty} G_\ve^-(t)=-\infty, 
\]
and
\[
|c(\wt M)-c(\wt M_\ve^{-})|<\ve.
\]
Here the functions 
\[
K(t):=-\frac{f''}{f}(t), \quad G_{\ve}^{-}(t):=-\frac{m_{\ve}^{-}{}^{\prime\prime}}{m_{\ve}^{-}}(t)
\] 
denote the radial curvature of $\wt M, \wt M_\ve^{-}$, respectively, and 
$||G_\ve^{-}-K||_2 := \sqrt{\int_{0}^{\infty} |G_{\ve}^{-} - K|^{2}\,dt}$.
\end{theorem}

\begin{remark} In Theorem \ref{thm1.4}, it is impossible to choose $\wt{M}_{\ve}^{-}$ as a 
von Mangoldt surface of revolution, when $K(t)$ is bounded from below. Here a von Mangoldt 
surface of revolution is, by definition, a model surface of revolution whose radial curvature 
is non-increasing on $[0, \infty)$.

\end{remark}

\section{Proof of Theorem \ref{thm1.3}}\label{sec:remarks}
By the same argument in the proof of \cite[Theorem 5.3]{KT2}, we have the next lemma.

\begin{lemma}\label{lem4.1}
Let $(M^{*}, p^{*})$ be 
a non-compact model surface of revolution with its metric 
$dt^2 +  m(t)^2d \theta^2$ satisfying the differential equation
$m''(t) + K (t) m(t) = 0$ with initial conditions $m(0) = 0$ and $m'(0) = 1$. 
If $M^{*}$ satisfies  
\[
\int^{\infty}_{0} t \,K (t) \,dt > - \infty
\]
and $K(t) \le 0$ on $[0, \infty)$, then $M^{*}$ admits a finite total curvature. 
\end{lemma}

\begin{lemma}{\bf (Model Lemma II)}\label{lem4.2}
Let $(\wt{M}, \tilde{p})$ denote 
a non-compact model surface of revolution with its metric 
$d\tilde{s}^2 = dt^2 + f(t)^2d \theta^2$ satisfying the differential equation $f''(t) + G (t) f(t) = 0$ 
with initial conditions $f(0) = 0$ and $f'(0) = 1$. 
If $\wt{M}$ admits a finite total curvature $c(\wt M)$ less than $2\pi$, 
then there exists a non-compact model surface of revolution $(M^{*}, p^{*})$ with its metric 
\begin{equation}\label{sec4-metric}
g^{*} = dt^2 +  m(t)^2d \theta^2
\end{equation}
satisfying the differential equation $m''(t) + G_{-} (t) m(t) = 0$ 
with initial conditions $m(0) = 0$ and $m'(0) = 1$ such that 
$M^{*}$ admits a finite total curvature. Here $G_{-} := \min \{G, 0\}$. 
\end{lemma}

\begin{proof}
Since $\wt{M}$ admits a finite total curvature, 
it follows from (5.2.6) in \cite{SST} that 
$\lim_{t \to \infty} f'(t) \in \R$ exists, and also from \cite[Theorem 5.2.1]{SST} that 
\[
2 \pi \lim_{t \to \infty} f'(t) = \lim_{t \to \infty} \frac{2 \pi f(t)}{t} = 2 \pi - c(\wt{M})
\]
holds. 
Since $- \infty < c (\wt{M}) < 2 \pi$ and 
\[
\lim_{t \downarrow 0} \frac{f(t)}{t} = 1,
\]
there exists a positive constant $\alpha$ such that 
\[
\frac{f(t)}{t} > \frac{1}{\alpha}
\]
on $(0, \infty)$.
Thus, 
\begin{equation}\label{lem3.3-1}
\int_{0}^{\infty} t\,G_{-} (t)\,dt 
\ge 
\alpha \int_{0}^{\infty}  f(t)G_{-} (t)\,dt.
\end{equation}
Since $c(\wt{M})$ is finite, 
\begin{equation}\label{lem3.3-2}
- \infty < \int_{\wt{M}} G_{-} \circ t \,d\wt{M} = 2\pi \int_{0}^{\infty}  f(t)G_{-} (t)\,dt.
\end{equation} 
By (\ref{lem3.3-1}) and (\ref{lem3.3-2}), 
\[
\int_{0}^{\infty} t\,G_{-} (t)\,dt 
> 
-\infty.
\]
Therefore, by Lemma \ref{lem4.1}, 
we get the non-compact model surface of revolution $(M^{*}, p^{*})$ with 
the metric (\ref{sec4-metric}) whose total curvature is finite. 
$\qedd$
\end{proof} 

\noindent
{\bf The proof of Theorem \ref{thm1.3}:}
By Lemma \ref{lem4.2}, we have a non-compact model surface of revolution $(M^{*}, p^{*})$ with 
its metric (\ref{sec4-metric}) whose total curvature is finite. 
Since $G \ge G_{-} = \min\{G, 0\}$, $(M^{*}, p^{*})$ is the reference surface to the $(M,p)$. 
Moreover, $(M^{*}, p^{*})$ has no pair of cut points in a sector $\wt{V}(\delta)$ for all 
$\delta \in (0, \pi]$, since $0 \ge G_{-}$. Therefore, by Theorem \ref{thm4.3}, 
$M$ is homeomorphic to the interior of a compact manifold with boundary.
$\qedd$

\section{Fundamental Lemmas}\label{sec2}

We need several lemmas for constructing a family of peculiar surfaces of revolution: 
Let $K:[0,\infty)\lra \R$ be a continuous function and let $f:[0,\infty)\lra \R$ 
be a solution of the following differential equation
\begin{equation}\label{eq2-1}
f''(t)+K(t)f(t)=0.
\end{equation}
Here we assume that the solution $f$ satisfies 
\begin{equation}\label{eq2-3}
f>0,
\end{equation}
on $(0,\infty)$, and
\begin{equation}\label{eq2-4}
\int^\infty_1 f(t)^{-2}dt<\infty.
\end{equation}

\begin{lemma}\label{lem2.1}
Let $G:[0,\infty)\lra \R$ be a continuous function
and let $m$ be the solution of the differential equation 
\begin{equation}\label{eq2-8}
m''(t)+G(t)m(t)=0
\end{equation} 
with initial conditions $m(0)=f(0)$ and $m'(0)=f '(0)$.
If $G-K$ has a compact
support in a bounded interval $[a,b] \subset[1,\infty)$, 
then, for any $t\geq a,$
\begin{equation}\label{eq2-17}
|\sigma(t)|\leq\int_a^tf(t)^{-2}|m'f-mf'|dt
\end{equation}
holds. Here we set
\[
\sigma(t):=\frac{m}{f}(t)-1.
\]
\end{lemma}

\begin{proof}
Since 
\[
\sigma'(t)=\frac{1}{f^2}(m'f-mf')(t)
\]
and $\sigma(t)=0$ on $(0,a]$,
we obtain
\[
\sigma(t)= \int_a^t\frac{1}{f^2}(m'f-mf')(t)\; dt
\]
and hence
\[
|\sigma(t)|\leq\int_a^t\frac{1}{f^2}|m'f-mf'|dt.
\]
$\qedd$
\end{proof}

\begin{lemma}\label{lem2.2}
If  $G$ and $m$ are the functions defined in Lemma \ref{lem2.1},
then,
\begin{equation}\label{eq2-10}
|(m'f-mf')(t)|\leq (\alpha(m)+1)\cdot||G-K||_2\cdot||f^2|_{[a,b]}||_2
\end{equation}
holds on $[0,\infty).$
Here we set  
\[
||G-K||_2:=\sqrt{\int_0^\infty |(G-K)(t)|^2dt}, \quad ||f^2|_{[a,b]}||_2:=\sqrt{\int_a^b\: f(t)^4 dt,}
\]
and $\alpha(m):=\sup_{t\geq 0}|\sigma(t)|$.
\end{lemma}

\begin{proof}
Since the case where $t\in[0,a]$ is trivial, we assume that 
$t>a.$
By the equations (\ref{eq2-1}) and (\ref{eq2-8}),
\begin{equation}\label{eq2-14}
(fm'-f'm)'(t)=(K-G)fm(t).
\end{equation} 
Hence,
\begin{equation}\label{eq2-15}
(fm'-f'm)(t)=(fm'-f'm)(b)
\end{equation}
holds for any $t\geq b,$
since $G=K$ on $[b,\infty).$
By (\ref{eq2-14}), we get
\[
|(fm'-f'm)|(t)\leq \int_a^t|K-G|f^2(|\sigma|+1)\: dt
=(\alpha(m)+1)\int_a^t|K-G|f^2\: dt.
\]
Now, it is clear from the Shwarz inequailty and (\ref{eq2-15}) that (\ref{eq2-10}) holds for any $t\geq 0$.
$\qedd$
\end{proof}

\begin{lemma}\label{lem2.3}
Set 
\[
C(f,a,b):=\int_a^\infty\frac{1}{f^2}\; dt\cdot||f^2|_{[a,b]}||_2 \ (> 0).
\]
If 
\[
C(f,a,b) < \frac{1}{||G-K||_2},
\]
then
\begin{equation}\label{eq2-22}
\alpha(m)\leq\frac{C(f,a,b)||G-K||_2}{1-C(f,a,b)||G-K||_2}.
\end{equation}
\end{lemma}

\begin{proof}
Since $\sigma(t)=0$ for any $t\in[0,a],$ it follows from (\ref{eq2-17}) and (\ref{eq2-10}) that
\[
\sup_{t\geq 0}|\sigma(t)|\leq C(f,a,b)\cdot||G-K||_2(\alpha(m)+1).
\]
Thus, it is clear that (\ref{eq2-22}) holds.
$\qedd$
\end{proof}

\begin{lemma}\label{lem2.4}
The equations
\begin{equation}\label{eq2-24}
\int_a^b|Gm-Kf|dt\leq(\alpha(m)+1)||G-K||_2\cdot||f|_{[a,b]}||_2
+\alpha(m)\int_a^b|f''|dt
\end{equation}
and 
\begin{equation}\label{eq2-24_2011_18}
\int_b^\infty|Gm-Kf|dt\leq \alpha(m)\int_b^\infty|f''|dt
\end{equation}
hold.
Hence, we get
\begin{equation}\label{eq2-26}
\int_0^\infty|Gm-Kf|dt\leq
\alpha(m)\int_a^\infty|f''|dt+(\alpha(m)+1)||G-K||_2\cdot||f|_{[a,b]}||_2.
\end{equation}
\end{lemma}
\begin{proof}
Since 
\begin{equation}\label{eq2-27}
(Gm-Kf)(t)=(G-K)(t)f(t)(\sigma(t)+1)+K(t)f(t)\sigma(t),
\end{equation}
we get, by the triangle inequality, 
\begin{equation}\label{eq2-28}
|Gm-Kf|(t)\leq(\alpha(m)+1)|G-K|(t)f(t)+\alpha(m)|K(t)f(t)|.
\end{equation} 
From the Shwarz inequality, it follows that
\begin{equation}\label{eq2-29}
\int_a^b|G-K|(t)f(t)dt\leq(\alpha(m)+1)||G-K||_2\cdot||f|_{[a,b]}||_2+\alpha(m)\int_a^b|Kf|dt.
\end{equation}
The equation (\ref{eq2-24}) is clear from (\ref{eq2-29}),
since
$Kf=-f''$ by (\ref{eq2-1}). Since $\supp(G-K)\subset[a,b],$ $G=K$ on $[b,\infty).$ 
Hence, 
$|Gm-Kf|(t)=|Kf\sigma(t)|\leq\alpha(m)|Kf|(t)$
on $[b,\infty)$ and
$Gm(t)=Kf(t) $ on $[0,a].$ Now, the equations (\ref{eq2-24_2011_18}) and (\ref{eq2-26}) are clear.
$\qedd$
\end{proof}
 
\begin{lemma}\label{lem2.5}
If $\alpha(m)<1,$ then $m(t)>0$ on $(0,\infty)$ and
 \begin{equation}\label{eq2-30}
 \int_1^\infty|f(t)^{-2}-m(t)^{-2}|dt\leq\frac{(2+\alpha(m))\alpha(m)}{(1-\alpha(m))^2}\int_a^\infty f(t)^{-2}dt.
 \end{equation}
 \end{lemma}
 \begin{proof}
Since $\sigma(t)\geq -\sigma(m)>-1$ for any $t\in[0,\infty),$ it is clear that $m(t)$ is positive on $(0,\infty)$. By definition, 
$m(t)^{-2}=(\sigma+1)^{-2}f(t)^{-2}$ holds. Hence, we get 
$$|f(t)^{-2}-m(t)^{-2}|=f(t)^{-2}|(\sigma+1)^{-2}-1|\leq\alpha(m)\cdot f(t)^{-2}\frac{|\sigma(t)|+2}{(1-|\sigma(t)|)^2}.$$  
Since the function $(x+2 )/ (1-x)^2$ is increasing on $[0,1),$
\begin{equation}\label{eq2-31}
|f(t)^{-2}-m(t)^{-2}|\leq\frac{\alpha(m)(2+\alpha(m))}{(1-\alpha(m))^2}f(t)^{-2}.
\end{equation}
Since $G=K$ on $[0,a],$ $f=m$ on $[0,a].$ Therefore, by (\ref{eq2-31}),
\[
\int_1^\infty |f(t)^{-2}-m(t)^{-2}|dt
=
\int_a^\infty|f(t)^{-2}-m(t)^{-2}|dt
\leq
\frac{\alpha(m)(2+\alpha(m))}{(1-\alpha(m))^2}\int_a^\infty f(t)^{-2}dt.
\]
$\qedd$
\end{proof}
 
\begin{proposition}\label{prop2.6}
Let $K:[0,\infty)\lra \R$ be a continuous function and let $f:[0,\infty)\lra \R$ be the solution of the differential equation of (\ref{eq2-1}) with initial conditions $f(0)=0$ and $ f'(0)=1$. 
Suppose that the solution $f$ satisfies (\ref{eq2-3}), (\ref{eq2-4}) and 
\[
 \int_0^\infty |f''(t)|dt<\infty.
\]
Then, for any $\ve>0$ and any bounded interval $(a,b)\subset[1,\infty),$
there exists $\delta>0$ such that for any continuous function $G:[0,\infty)\lra \R$ satisfying
$\supp(G-K)\subset[a,b]$ and 
\[
||G-K||_2:=\sqrt{\int_0^\infty|G-K|^2dt}<\delta,
\]
the solution $m$ of the differential equation 
$m''(t)+G(t)m(t)=0$ with initial conditions
$m(0)=0$ and $m'(0)=1$,
satisfies 
\begin{equation}\label{eq2-34}
\int_0^\infty |Gm(t)-Kf(t)|dt<\ve,
\end{equation}
and
\begin{equation}\label{eq2-35}
\int_1^\infty|m(t)^{-2}-f(t)^{-2}|dt<\ve.
\end{equation}
\end{proposition}
 
\begin{proof}
Let $\ve$ be an arbitrarily fixed number. 
Here we choose a positive 
number $\delta_1\in(0,1/C(f,a,b))$ in such a way that
 \begin{equation}\label{eq2-36}
 \frac{\delta_1}{1-C(f,a,b)\delta_1}||f|_{[a,b]}||_2<\frac{\ve}{2}
 \end{equation}
 and
 \begin{equation}\label{eq2-37}
 \frac{C(f,a,b)\delta_1}{1-C(f,a,b)\delta_1}\int_a^\infty|f''|dt<\frac{\ve}{2}
 \end{equation}
 hold. Then, it follows from Lemma {\ref{lem2.3}}, (\ref{eq2-36}), and (\ref{eq2-37})  
 that for any continuous function $G:[0,\infty)\lra \R$ satisfying 
 $\supp(G-K)\subset[a,b]$ and $||G-K||_2<\delta_1,$ the solution $m$ satisfies 
\[
\alpha(m)\int_a^\infty|f''|dt<\frac{\ve}{2}
\]
and 
\[
(\alpha(m)+1)||G-K||_2\cdot||f|_{[a,b]}||_2<\frac{\ve}{2}.
\]
Now, the equation (\ref{eq2-34}) 
is clear from (\ref{eq2-26}). Moreover, by the equations (\ref{eq2-30}) and (\ref{eq2-22}), 
there exists $\delta\in(0,\delta_1]$ such that   for any continuous function $G:[0,\infty)\lra \R$ 
satisfying $\supp(G-K)\subset[a,b]$ and $||G-K||_2<\delta<\delta_1,$ the solution $m$ 
satisfies (\ref{eq2-34}) and (\ref{eq2-35}).$\qedd$
\end{proof}

The following proposition is clear from Lemmas \ref{lem2.1}, \ref{lem2.2}, \ref{lem2.3}, \ref{lem2.5} and the proof of Proposition \ref{prop2.6}.

\begin{proposition}\label{prop2.7} 
Let $K:[0,\infty)\lra \R$ be a continuous function and let $f:[0,\infty) \lra \R$ 
be the solution of the differential equation of (\ref{eq2-1}) with initial conditions $f(0)=0$ and $f'(0)=1$. Suppose that the solution $f$ satisfies (\ref{eq2-3}) and (\ref{eq2-4}). 
Then, for any $\ve>0$ and any bounded interval $(a,b)\subset[1,\infty),$
there exists $\delta>0$ such that for any continuous function $G:[0,\infty)\lra \R$ satisfying
$\supp(G-K)\subset[a,b]$, and 
\[
||G-K||_2:=\sqrt{\int_0^\infty|G-K|^2dt}<\delta,
\]
the solution $m$ of the differential equation
$m''(t)+G(t)m(t)=0$ with initial conditions
$m(0)=0$ and $m'(0)=1$, satisfies the equation (\ref{eq2-35}).
\end{proposition}

\section{The Construction of a Peculiar Model} \label{construction}

\noindent
{\bf The proof of Theorem \ref{thm1.4}:} 
From the isoperimetric inequalities (see \cite[Theorem 5.2.1]{SST}) 
and the l'H\^opital's theorem, it follows that 
\[
2\pi\lim_{t\to\infty}f'(t)=\lim_{t\to\infty}\frac{2\pi f(t)}{t}=2\pi-c(\wt M).
\]
Hence, the property $c(\wt M)<2\pi$ implies that 
\[
\lim_{t\to\infty}\frac{f(t)}{t}=\lim_{t\to\infty}f'(t)>0.
\]
In particular, 
\[
\int_1^\infty f(t)^{-2}dt<\infty.
\]
Since $c(\wt M)$ is finite,
 $$2\pi\int_0^\infty|K(t)|f(t)dt<\infty.$$
 This is equivalent to 
 $$\int_0^\infty |f''(t)|dt<\infty.$$
By applying Proposition \ref{prop2.6} for the interval $(3/2,5/2)$ and $\ve / 9\pi,$ 
we may find a smooth function $G_1 : [0,\infty)\lra \R$ such that 
$||G_1-K||_2< \ve / 3^2$, 
$K\geq G_1$ on $[0,\infty)$, 
$\supp(K-G_1)\subset (3/2,5/2)$,
\[
\int_0^\infty |m_1''-f''|dt<\frac{\ve}{9\pi}, \quad 
\int_1^\infty|m_1^{-2}-f^{-2}|dt<\frac{\ve}{9\pi}<\frac{\ve}{9},
\]
and
$\min\{G_1(t) ; 3/2\leq t \leq 5/2\}\leq -1$.
Here $m_1$ denotes the solution
$m_1''+G_1m_1=0$ 
with initial conditions $m_1(0)=0$ and $m_1'(0)=1.$ 
By applying Proposition \ref{prop2.6}, it is easy to define a sequence 
of smooth functions $\{G_k:[0,\infty)\lra \R\}_{k\geq 0},$ where $G_0=K,$ satisfying
$||G_k-G_{k-1}||_2< \ve / 3^{k+1}$, $G_{k-1}\geq G_k$ on $[0,\infty)$,
$\supp(G_k-G_{k-1})\subset(2k-1/2,2k+1/2)$, 
\[
\int_0^\infty|m_k''-m_{k-1}''|dt<\frac{\ve}{3^{k+1}\pi},
\quad 
\int_1^\infty|m_k^{-2}-m_{k-1}^{-2}|dt<\frac{\ve}{3^{k+1}},
\]
and $\min\{G_k(t); 2k-1/2\leq t\leq 2k+1/2\} \leq -k$.
Here $m_k$ denotes the solution of
$m_k''+G_km_k=0$ with initial conditions $m_k(0)=0$ and $m_k'(0)=1$.
We define ${m_\ve}(t):=\lim_{k\to\infty}m_k(t)$ and ${G_\ve}(t):=\lim_{k\to\infty}G_k(t)$. 
It is easy to check that $m_\ve(t) $ is the solution of 
${m_\ve}''+{G_\ve }(t){m_\ve}(t)=0$ with initial conditions 
${m_\ve}(0)=0$ and ${m_\ve}'(0)=1$.
Furthermore, the function $m_\ve$ and $G_\ve{}$ satisfy
\begin{equation}\label{eq3-23}
\int_0^\infty |{m_\ve}''  - f''|dt \leq \frac{\ve}{3\pi},
\end{equation}
$\liminf_{t\to\infty} G_\ve(t)=-\infty$, 
$K\geq G_\ve \ {\rm on} \ [0,\infty)$, 
and 
$||G_\ve -K||_2\leq \ve / 3<\ve$. 
The equation (\ref{eq3-23}) implies that
\[
|c(\wt{M}_\ve^{-})-c(\wt M)|\leq 2\pi\int_0^\infty|{m_\ve}''-f''|dt\leq
\frac{2\ve}{3}<\ve,
\]
where $\wt{M}_\ve^{-}$ is a non-compact model surface of revolution such that 
$\wt{M}_\ve^{-} :=(\R^2,dt^2+m_\ve^-(t)d\theta^2)$ and $m_\ve^-(t):= m_\ve(t)$.
$\qedd$

\bigskip

The proof of the following theorem is  similar to that of the theorem above.

\begin{theorem}\label{thm3.1}
Let $\wt M:=(\R^2,dt^2+f(t)^2d\theta^2)$ denote a non-compact model 
surface of revolution which admits a finite total curvature $c(\wt M)$ less than $2\pi.$ 
Then, for any $\ve>0,$ there exists a non-compact model surface of revolution 
$\wt M_\ve^+:=(\R^2,dt^2+m_\ve^+(t)d\theta^2)$ 
such that 
$G_\ve^+\geq K$ on $[0,\infty)$, $||G_\ve^{+}-K||_2<\ve$, 
$\limsup_{t\to \infty}G_\ve^+(t)=\infty$, and
$|c(\wt M)-c(M_\ve^{+})|<\ve$, where we denote by 
$K :=- f'' / f, G_{\ve}^{+}:=- m_{\ve}^{+}{}^{''} / {m_\ve}^+$ 
the radial curvature of $\wt M, \wt M_\ve^+$ respectively.
\end{theorem}

\begin{corollary}\label{cor3.1}
Let $\wt M:=(\R^2,dt^2+f(t)^2d\theta^2)$ denote a non-compact model 
surface of revolution which satisfies (\ref{eq2-3}) and (\ref{eq2-4}).
 Then, for any $\ve>0,$ there exist non-compact model surfaces of revolution 
 $\wt M_\ve^+:=(\R^2,dt^2+m_\ve^+(t)d\theta^2)$ and $\wt M_\ve^-:=(\R^2,dt^2+m_\ve^-(t)d\theta^2)$ such that $G_\ve^+\geq K\geq G_\ve^-$ on $[0,\infty)$, $||G_\ve^{*}-K||_2<\ve$, and 
$\int_1^\infty|f(t)^{-2}-{m_\ve^*}^{-2}|dt <\ve$. Here $*=\pm 1.$ 
\end{corollary}

\bigskip

\begin{center}
Minoru TANAKA $\cdot$ Kei KONDO 

\medskip
Department of Mathematics\\
Tokai University\\
Hiratsuka City, Kanagawa Pref.\\ 
259\,--\,1292 Japan

\medskip

{\small
$\bullet$\,our e-mail addresses\,$\bullet$

\medskip
\textit{e-mail of Tanaka}:

\medskip
{\tt tanaka@tokai-u.jp}

\medskip 
\textit{e-mail of Kondo}:

\medskip
{\tt keikondo@keyaki.cc.u-tokai.ac.jp}

}
\end{center}

\end{document}